\documentclass[11pt]{amsart}
\usepackage[utf8]{inputenc}
\usepackage{amsmath}
\usepackage{amssymb, amsthm, amsfonts, amscd, MnSymbol, url}
\usepackage[top=1.25in, bottom=1.25in, outer=1.26in, inner=1.26in, heightrounded, marginparwidth=1.2in, marginparsep=.15cm]{geometry}
\usepackage{amsxtra}     
\usepackage{epsfig}
\usepackage{verbatim}
\usepackage{graphicx}
\usepackage[all]{xypic}
\usepackage{todonotes}
\usepackage{marginnote}

\usepackage{mathrsfs}

\usepackage{hyperref}

\usepackage{enumerate}   

\input xy
\xyoption{all}

\def\bA{\mathbb{A}}
 
\def\bC{\mathbb{C}}

\def\bQ{\mathbb{Q}}

\def\bZ{\mathbb{Z}}

\def\cO{\mathcal{O}}
\def\CO{\cO}

\def\fp{\mathfrak{p}}

\def\GL{\operatorname{GL}}
\def\Gal{\operatorname{Gal}}
\def\id{\operatorname{id}}

\newcommand{\cmfield}{K}
\newcommand{\realfield}{\cmfield^+}
\newcommand{\OK}{\CO_\cmfield}

\newcommand{\IC}{\bC}
\newcommand{\ZZ}{\bZ}
\newcommand{\IQ}{\bQ}
\newcommand{\adeles}{\bA}
\newcommand{\isomto}{\overset{\sim}{\rightarrow}}

\newcommand{\subDelta}{{_\Delta}}
\newcommand{\subGamma}{{_\Gamma}}

\newcommand{\central}{s_0}
\newcommand{\critical}{s_1}
\newcommand{\normcharacter}{\mathcal{N}}
\newcommand{\fixedchar}{\mathscr{C}}
\newcommand{\measurecritical}{0}
\newcommand{\fixedconductor}{\mathfrak{m}}

\newcommand{\anticycchar}{\mathscr{A}}

\newtheorem{thm}{Theorem}
\numberwithin{thm}{section}
\newtheorem{cor}[thm]{Corollary}

\newtheorem{lem}[thm]{Lemma}
\newtheorem{prop}[thm]{Proposition}

\theoremstyle{definition}

\newtheorem{cond}[thm]{Condition}
\newtheorem{example}[thm]{Example}
 
\theoremstyle{remark}

\newtheorem{rmk}[thm]{Remark}

\numberwithin{equation}{section}

\title[Twisted $L$-functions]{Applications of nonarchimedean developments\\
 to  archimedean nonvanishing results\\ 
 for twisted $L$-functions}

\author[E. E. Eischen]{E. E. Eischen}
\thanks{This research was partially supported by NSF Grant DMS-1559609 and NSF CAREER Grant DMS-1751281.}

\address{E. Eischen\\
Department of Mathematics\\
University of Oregon\\
Fenton Hall\\
Eugene, OR 97403-1222\\
USA}
\email{eeischen@uoregon.edu}

\begin{document}
\bibliographystyle{halpha}

\newpage
\setcounter{page}{1}
\maketitle
\vspace{-0.2in}
\begin{abstract}

We prove the nonvanishing of the twisted central critical values of a class of automorphic $L$-functions for twists by all but finitely many unitary characters in particular infinite families.  While this paper focuses on $L$-functions associated to certain automorphic representations of unitary groups, it illustrates how decades-old nonarchimedean methods from Iwasawa theory can be combined with the output of new machinery to achieve broader nonvanishing results.  In an appendix, which concerns an intermediate step, we also outline how to extend relevant prior computations for $p$-adic Eisenstein series and $L$-functions on unitary groups to the case where primes dividing $p$ merely needs to be unramified (whereas prior constructions required $p$ to split completely) in the associated reflex field. 
\end{abstract}


\section{Introduction}

From the proof of Dirichlet's theorem on arithmetic progressions of prime numbers \cite{dirichlet} to Shimura's proof of algebraicity of certain values of $L$-functions \cite{shimura-hilbert} to the class number formula for number fields to the Birch and Swinnerton-Dyer conjecture and Bloch--Kato conjectures \cite{BK, fontaine-PR, SkiUrb, CGH}, 
nonvanishing of $L$-functions at certain points has significant consequences.  

This paper concerns the (non)vanishing at the {\it central critical value} of $L$-functions associated to certain cuspidal automorphic representations $\pi$ of a unitary group twisted by Hecke characters $\chi$, but the methods (outside of the Appendix) are applicable to $L$-functions associated to a wider class of data.  (The Appendix, however, only applies to unitary groups, as it explains how several very recent developments enable removal of certain technical conditions, in particular removal of the requirement that $p$ split, from earlier nonarchimedean results for automorphic forms and $L$-functions associated to unitary groups.)  In analogue with the Birch and Swinnerton-Dyer conjecture, the Bloch--Kato conjectures predict the algebraic meaning (in terms of ranks of certain modules $M$ associated to $\pi$) of the vanishing at critical points (i.e. as defined in \cite{deligne}, so that the gamma factors occurring in the functional equation for the $L$-function do not have poles).  In the framework of the Bloch--Kato conjectures, twists by $\chi$ correspond (conjecturally) to certain $\chi$-eigenspaces of $M$ (viewed as a Galois module, with $\chi$ viewed as a character on a Galois group).

The values at the central critical point are generally particularly challenging to study.  For example, the $L$-functions associated to tempered cuspidal automorphic representations on unitary groups are known to be nonzero at critical points to the right of the central critical point, but the behavior at the central critical point largely remains mysterious.  Methods in analytic number theory often focus on proving results about averages of central values of twists of $L$-functions by characters $\chi$ as those characters vary in certain families.

Even though the results in this paper concern $\IC$-valued meromorphic $L$-functions and one often uses (archimedean) analytic methods to study the behavior of the twisted central critical values, the proofs in this paper mostly exploit $p$-adic methods, building on an idea of R.\ Greenberg in \cite{gr1}.  This paper illustrates how decades-old methods from Iwasawa theory can be combined with the output of new machinery to achieve nonvanishing results about functions arising in analytic number but which current (archimedean) analytic number theoretic methods cannot address (at least, as experts have indicated to me).

As such, we aim in the present paper to present the results and techniques in a reasonably accessible manner to someone who is interested in the $\IC$-valued $L$-functions under consideration but not necessarily an expert in the $p$-adic theory.  While many of the references from which we extract key ingredients require a background in Iwasawa theory, unitary Shimura varieties, or related topics, the reader of the present paper need not be an expert in them.

\subsection{Variation of Dirichlet characters and more}
B.\ Mazur distinguished in \cite{mazur-warsaw} between two natural types of variation of $\chi$: {\it horizontal variation}, with $\chi$ ranging over characters of fixed order (varying the conductor, e.g.\ \cite{BFH, FrHo, OnSk}), and {\it vertical variation}, varying $\chi$ over finite order characters with conductor divisible by some fixed finite set of primes.  This paper concerns vertical variation, as we vary over conductors dividing $p^\infty$ (or, more generally,  $p^\infty\fixedconductor$ for some fixed conductor $\fixedconductor$ prime to $p$) for an odd prime $p$.

In the setting of this paper, one could consider a third type of variation, which we leave for future work: For $\chi$ fixed, we may vary $\pi$ analogously to vertical variation of $\chi$, at least if $\pi$ is {\it ordinary} at a prime $p$ and parametrized by tuples of characters like in \cite{HELS}.  In the inspiration for the present paper, Greenberg's vertical variation of $\chi$ in $L$-functions associated to certain CM Hecke characters \cite{gr1}, this third type of variation does not exist (or at least, amounts to vertical variation of $\chi$), since representations of $\GL_1$ are characters.

\subsection{Main results}
The main results concern the nonvanishing of the twisted central critical values of certain automorphic $L$-functions for twists by all but finitely many unitary characters in particular infinite families.  Some of the propositions and lemmas proved en route to these results are also interesting and,  if one assumes certain conjectures of J.\ Coates and B.\ Perrin-Riou \cite{coates, CoPR}, are applicable to $L$-functions attached to a wider class of data.  Indeed, as new integral representations of $L$-functions continue to become available, constructions of associated $p$-adic $L$-functions will likely enable immediate extension of the main results of the present paper to other representations (including via current, ongoing projects of the author).  While many of the results here can be extended to other groups, the Appendix concerns only unitary groups.

Let $\pi$ be a tempered cuspidal automorphic representation of a general unitary group preserving a hermitian form on a vector space over a quadratic imaginary field $K$, and let $\chi$ be a Hecke character on the id\`eles of $K$.  Fix an odd prime $p$ that splits in $K$.\footnote{In the Appendix, building on very recent developements \cite{brasca-rosso, EiMa, liuJussieu, liu-rosso}, we explain the extent to which we can remove this condition on $p$ and replace it with the condition that $p$ is merely unramified.}

Note that the central critical point $\central$ for $L(s, \pi, \chi)$ depends on the pair $(\pi, \chi)$ and is the same for $(\pi, \chi\chi')$ for any finite order character $\chi'$, since only the infinity type of $\chi$ affects the value of $\central$.  A Hecke character of type $A_0$ (in the sense of \cite{weil}) may be identified with a character on $G_K=\Gal(\bar{K}/K)$, which we also denote by $\chi$.  

As discussed in Section \ref{background-section}, there are certain extensions $K^\pm$ of $K$ with Galois group $\Gamma^\pm\cong\ZZ_p$, with $K^+$ denoting the {\it cyclotomic} $\ZZ_p$-extension (the unique extension of $K$ with Galois group isomorphic to $\ZZ_p$ contained in the extension obtained by adjoining all $p$-power roots of unity to $K$) and $K^-$ denoting the {\it anti-cyclotomic} $\ZZ_p$-extension (the compositum of the $p$-power cyclic extensions of $K$ that are dihedral over $\IQ$).  Characters of $\Gal(\bar{K}/K)$ that factor through $\Gamma^+$ (resp.\ $\Gamma^-$) are called {\it cyclotomic} (resp.\ {\it anti-cyclotomic}).  We let $\Gamma$ be the Galois group of the compositium $K^+K^-$ over $K$.

On the Galois group of the compositum of the cyclotomic $p$-power extensions of $K$, identified with $\ZZ_p^\times$ through the Galois action on the $p$-power cyclotomic extensions of $K$, there is a unique $(\ZZ/p\ZZ)^\times$-valued character $\omega$ (the Teichm\"uller character, which we view as a Hecke character) with the property that $\omega(a)\equiv a\mod p$ for each $a\in\ZZ_p^\times.$

\begin{thm}[Theorem \ref{ord-Thm}]\label{ord-Thm-intro}
Let $\pi$ be a representation meeting Condition \ref{cond-ord} (e.g.\ one of the infinitely many representations in \cite{HELS}, extended further by Section \ref{app-section} of the present paper).  For each critical point $\critical$ to the right of the central critical point $\central$ for $(\pi, \chi)$
\begin{align}\label{extra-term-intro}
L(\central, \pi, \chi\omega^{\central-\critical}\psi)\neq 0
\end{align}
for all but finitely many (of the infinitely many) finite order cyclotomic Hecke characters $\psi$ of $\Gamma^+$.  Moreover, for all but finitely many of those $\psi$, 
\begin{align*}
L(\central, \pi, \chi\omega^{\central-\critical}\psi\psi')\neq 0
\end{align*}
for all but finitely many unitary anti-cyclotomic characters $\psi'$ of $\Gamma^-$.
\end{thm}
For simplicity, we have assumed in Theorem \ref{ord-Thm-intro} that $\chi$ has conductor dividing $p^\infty$.  For more general conductor, there is an additional finite order character (given precisely in Theorem \ref{ord-Thm}) by which we must twist the $L$-function in Equation \eqref{extra-term-intro}.

The condition at the beginning of Theorem \ref{ord-Thm-intro} is satisfied, for example, if $\chi$ is of type $A_0$ and $\pi$ is {\it ordinary} at $p$ (as a consequence of the main results of \cite{HELS}).  All tempered cuspidal automorphic representations ordinary at $p$ satisfy Condition \ref{cond-ord}, by the main results of \cite{HELS}.  Furthermore, this can be extended (via the Appendix in Section \ref{app-section}, employing \cite{brasca-rosso, EiMa, liuJussieu, liu-rosso}) to certain {\it $P$-ordinary} representations, For generalizations to $L$-functions attached to other data, we have Theorem \ref{cond-Thm-intro}, which relies on a weaker condition.
Note that for tempered cuspidal automorphic representations $\pi$ on unitary groups, for each $\chi$ of type $A_0$, $L(s, \pi, \chi)\neq 0$ for $s$ to the right of the central critical point for $(\pi, \chi)$.  Therefore, with sufficient information about the exceptional zeroes (equivalently, the form of the modified Euler factors) of an associated $p$-adic $L$-function, we could obtain the stronger theorem.  (In the statement, $\Delta\times\Delta_\fixedconductor$ is a certain finite group and $\normcharacter$ is the norm character.)

\begin{thm}[Theorem \ref{cond-Thm}]\label{cond-Thm-intro}
For each $(\pi, \fixedchar)$ meeting Condition \ref{cond-pLfcnexists} and type $A_0$ Hecke character $\chi$ of $\adeles_K^\times$ such that $\central$ is central for $(\pi, \chi)$ and such that $\chi\normcharacter^{-\central}_{\Delta\times\Delta_\fixedconductor}=\fixedchar$, there is a finite order Hecke character $\chi'$ of $\Gamma$ such that such that at the central critical point $\central$ for $(\pi, \chi)$, 
\begin{align}\label{cond-Thm-equ-intro}
L(\central, \pi, \chi\chi')\neq 0
\end{align}
and
\begin{align*}
L(\central, \pi, \chi\chi'\psi)\neq 0
\end{align*}
for all but finitely many unitary cyclotomic (resp.\ anti-cyclotomic) characters $\psi$ of $\Gamma^+$ (resp.\ $\Gamma^-$).  Furthermore, for all but finitely many such $\psi$, $L(\central, \pi, \chi\chi'\psi\psi')\neq 0$ for all but finitely many unitary anti-cyclotomic (resp.\ cyclotomic) characters $\psi'$ of $\Gamma^-$ (resp. $\Gamma^+$).
\end{thm}

\begin{rmk}
Assuming the Langlands conjectures, our nonvanishing results immediately imply new nonvanishing results for $\GL_n$.  More precisely, the Langlands conjectures (see, for example, \cite{arthur, cogdell}) predict that associated (via Langlands functoriality) to the cuspidal automorphic representation $\pi$ of a unitary group of rank $n$, there is an automorphic representation of $\GL_n$ whose $L$-function agrees with the $L$-function associated to $\pi$.  Langlands functoriality also respects tensor products, so our results on twisted $L$-functions also carry over.  If the cuspidal automorphic representation $\pi$ is {\it ordinary} at $p$, then it meets Condition \ref{cond-ord}.  As explained in \cite[Section 8]{H98}, ordinariness (what Hida calls {\it near-ordinariness}) at $p$ is preserved under functoriality.  (Ordinariness is determined by the action of certain Hecke operators that are unnecessary in the present paper.)  A more general class of representations of unitary groups that includes those in the Appendix Section \ref{app-section} ({\it $P$-ordinary} representations, which also satisfy Condition \ref{cond-ord}) has an analogous property preserved under functoriality, by  \cite[Section 8]{H98}.
\end{rmk}

\subsection{Approach}
As indicated above, this paper illustrates how decades-old methods in Iwasawa theory can be combined with the output of new machinery to prove results about functions arising in analytic number but for which no known proof via (archimedean) analytic number theoretic methods currently exists.  

\subsubsection{Existence of $p$-adic $L$-functions}
Readers unfamiliar with $p$-adic measures and Iwasawa algebras are encouraged to consult \cite[\S 12.1-12.2 and \S7.1-7.2]{washington} alongside this paper.
Most of the paper relies on the existence of a $p$-adic $L$-function (realized as a $p$-adic measure and as an element of an Iwasawa algebra over a $p$-adic ring) whose values $p$-adically interpolate (an appropriate normalization of) $L(\critical, \pi, \chi)$ at critical points $\critical$ and Hecke characters $\chi$.  

The machinery in \cite{HELS} allows the construction of $p$-adic $L$-functions for $\pi$ {\it ordinary} at $p$, and this is extended further in Section \ref{app-section} (Appendix) of the present paper, by building on very recent developments (\cite{brasca-rosso, EiMa, liuJussieu, liu-rosso}).  To prove some of the results of the present paper concerning cyclotomic characters, it is crucial for the $p$-adic $L$-function to interpolate among different critical values.  For this reason, the earlier construction of associated $p$-adic $L$-functions proposed in \cite{HLS} is insufficient for completing the work in the present paper.  Instead, one must work with more general $p$-adic $L$-functions, which employ the Eisenstein measure in \cite{apptoSHL} and differential operators from \cite{EDiffOps} that enable interpolation among all the critical values.

Coates and Perrin-Riou have conjectured the existence of a wide class of $p$-adic $L$-functions, attached to ordinary motives \cite{coates, CoPR}.  Assuming their conjectures, one can extend most of the work in this paper to their setting, as we use almost nothing specific to unitary groups in the purely $p$-adic portions this paper, except for the crucial fact that we have already constructed the requisite $p$-adic $L$-functions in that case.  In fact, we have stated most of the propositions and lemmas so that they can easily be modified to accommodate other cases where the requisite $p$-adic $L$-functions exist.

While the construction in \cite{HELS} (and \cite{apptoSHL, EDiffOps}) relies on the geometry of Shimura varieties (extending \cite{kaCM} to unitary groups), the method of construction of the $p$-adic $L$-functions is irrelevant to the nonvanishing results in this paper.  What matters is their existence as $p$-adic measures and, equivalently, elements of Iwasawa algebras identified with power series rings. 

\subsubsection{Structure of $p$-adic power series rings and properties of $p$-adic measures}  
Building on an idea of Greenberg in \cite{gr1}, we exploit the structure of the Iwasawa algebra (as a power series ring) in which the $p$-adic $L$-function is realized.  In particular, we use the Weierstrass Preparation Theorem (\cite[Chapter VII Section 4]{bourbaki1998commutative}) in the proof of Proposition \ref{all-or-none}, which states, roughly, for each critical value that either all the cyclotomic (resp.\ anti-cyclotomic) twists of the $L$-function vanish or only finitely many do.  (The proof actually shows this for the $p$-adic $L$-function, and hence, one needs to know something about the relationship between the $L$-function and the associated $p$-adic $L$-function.)  One could also apply P.\ Monsky's results on the structure of $p$-adic power series rings, for example \cite[Theorem 2.6]{monsky}, in our context to obtain a more refined description of the set of zeroes of these $p$-adic $L$-functions.  

Given Proposition \ref{all-or-none}, it is sufficient to prove that at least one value of the $p$-adic $L$-function does not vanish.  For this, we actually show the nonvanishing of infinitely many values.  While our approach to this task necessarily invokes entirely different methodology that used by from Greenberg in his setting, we both rely on relating limits (archimedean limits in his case, $p$-adic in ours) of linear combinations of the $L$-values in question to a different $L$-value known to be nonzero.  There is no clear way to extend Greenberg's approach, an intricate real-analytic argument involving Abel means, to our situation.  Likewise, there is no clear way to extend our approach, a $p$-adic argument requiring the existence of a nonvanishing critical value to the right of the central critical point, to his situation. 

More precisely, using the fact that 
\begin{align}\label{norm-shift}
L(\critical, \pi, \chi) = L(\central, \pi, \chi\normcharacter^{-\critical+\central}),
\end{align}
(where $\normcharacter$ denotes the norm character), $p$-adically approximating $\normcharacter$ by linear combinations of finite order cyclotomic characters, and using the formulation of the $p$-adic $L$-function as a $p$-adic measure, we achieve Lemma \ref{existence-lemma}, which states (roughly) that infinitely many twists by cyclotomic characters do not vanish, so long as {\it some} critical (not necessarily central) value is known to vanish.  For this, we need a value of $L(s, \pi, \chi)$ at a critical point to the right of the central critical point not to vanish and not to correspond to an exceptional zero of the associated $p$-adic $L$-function.  For this, it is crucial that we know something about the $L$-functions with which we work, and merely taking an $L$-function whose values are interpolated by a $p$-adic $L$-function (without knowing something about its zeroes and exceptional zeroes of the associated $p$-adic $L$-function) will not, by itself, suffice.

 \subsection{Related directions}\label{related-directions}
 The inspiration for the approach in this paper comes from R.\ Greenberg's work in \cite{gr1, gr2}.  For a different class of $L$-functions for which the construction of $p$-adic $L$-functions follows from different techniques, F.\ Januszewski also recently independently obtained related nonvanishing results \cite{janu}.  
  As the theorems (but neither the propositions nor the lemmas) in this paper require nonvanishing at a critical point to the right of the central critical point, our theorems do not naturally extend to modular forms of weight $2$ (elliptic curves).  As noted above, though, they may extend to certain $L$-functions associated to other data.  More immediately, one can use \cite{ch, ch2, chenevier-harris, harris-takagi} to associate Galois representations $\rho_\pi$  (including in $p$-adic families) to $\pi$ and reframe the results from this paper in terms of Galois representations.  We have restricted to CM fields of degree $2$ in this paper, but one can extend the results to CM fields of arbitrary finite degree (at the cost of extra notation).  We allow more general conductor than Greenberg, but this does not cost us anything except some extra notation.
 
 \subsubsection{Additional approaches to closely related nonvanishing results}\label{additional-approaches}
Among the other related nonvanishing results in the literature, the closest appear to be D.\ Rohrlich's results for $\GL_2$ proved in \cite{rohrlich1, rohrlich2, rohrlich3, rohrlich4} and the strategy introduced by F.\ Shahidi for Rankin--Selberg products in \cite{shahidi}.  As Rohrlich notes in the introduction to \cite{rohrlich1}, it is not clear how to extend his approach employed in his deep arguments to higher rank groups.  Shahidi's idea is to use the Langlands--Shahidi method (also exploited in \cite{GJR}) to realize the reciprocals of certain $L$-functions in the constant terms of holomorphic Eisenstein series and conclude the nonvanishing of the $L$-function from the lack of poles of a holomorphic function.  While Shahidi's approach might be adaptable to our setting, we have not attempted that route.  In yet another direction, using the theory of endoscopy, D. Jiang and L. Zhang have recently proved related nonvanishing results under certain conditions in \cite{jiang-zhang}.

\subsubsection{The algebraic meaning of nonvanishing}
The reader who compares \cite[Propositions 3 and 4]{gr1} with Proposition \ref{all-or-none} will note that we consider more general classes of characters but do not address the {\it critical divisor} that appears in Greenberg's work.  The critical divisor enables Greenberg to relate the vanishing of $L$-functions (at Gr\"ossencharacters associated to CM elliptic curves) to the structure of CM elliptic curves.  For the Hecke characters corresponding to those in his work (see Remark \ref{comparison-rmk} for the precise connection between his statements and ours), his claims about divisibility by the critical divisor carry over immediately.  A step toward interpreting the algebraic meaning of our more general vanishing statements would be to extend the realization of critical divisors in our setting.

\subsection{Acknowledgements}
I am grateful to Michael Harris for insightful discussions and suggestions, as well as encouragement to write this paper.  I also thank Jianshu Li and Aaron Pollack for helpful tips about local Godement--Jacquet style integrals, Rob Harron for a helpful discussion about exceptional zeroes, and David Rohrlich for a helpful suggestion to consider implications (via Langlands functoriality) for $L$-functions associated to representations of $\GL_n$.

\section{Background and Notation}\label{background-section}
Fix an algebraic closure $\bar{\IQ}$ of $\IQ$.  Let $K\subset \bar{\IQ}$ be a quadratic imaginary extension of $\IQ$, let $\cO_K$ denote the ring of integers in $K$, and let $p$ be an odd rational prime that splits in $K$.  Write $(p)=\fp\fp^*$ for the prime factorization of $(p)$ in $\cO_K$. 
 Let $c$ denote complex conjugation.  Given an element $a\in\IC$, we set $\bar{a}:=c(a)$.

Denote by $\IC_p$ the completion of an algebraic closure of $\IQ_p$, and let $\cO$ denote the ring of integers in $\IC_p$.  Fix an embedding $\sigma_\fp: \bar{\IQ}\hookrightarrow\IC_p$ inducing the $\fp$-adic valuation on $K$, and fix an embedding $\sigma_\infty:\bar{\IQ}\hookrightarrow\IC$.  Via $\sigma_\fp$ (resp.\ $\sigma_\infty$), we identify $\bar{\IQ}$ with its image in $\IC_p$ (resp.\ $\IC$).

\subsection{Important extensions and Galois groups}
Let $K_\infty^-$ denote the anti-cyclotomic $\ZZ_p$-extension of $K$, so $K_\infty^-=\cup_{n\geq 0}K_n^-$ with $K_n^-$ cyclic of degree $p^n$ over $K$ and $\Gal(K_n^-/\IQ)$ dihedral for all $n$.  For each element $g\in\Gal(K_n^-/K)$, we have $c\circ g \circ c^{-1} = g^{-1}$.
Let $K_\infty^+$ denote the cyclotomic $\ZZ_p$-extension of $K$, so $K_\infty^+=\cup_{n\geq 0}K_n^+$ with $K_n^+$ cyclic of degree $p^n$ contained in the $p^{n+1}$th cyclotomic extension $K(\mu_{p^{n+1}})$, where $\mu_{p^{n+1}}$ denotes the group of $p^{n+1}$-power roots of unity, and $K_\infty^+\subseteq K(\mu_p^\infty)$, with $\mu_p^\infty$ the group of all $p$-power roots of unity.  Let $K_\infty = K_\infty^+K_\infty^-$.  So $K_\infty$ is the compositum of all the $\ZZ_p$-extensions of $K$.  Define $\Gamma := \Gal\left(K_\infty/K\right)$, $\Gamma^-=\Gal(K_\infty^-/K)$, and $\Gamma^+=\Gal(K_\infty^+/K)$.  So $\Gamma^-\cong\ZZ_p$, $\Gamma^+\cong\ZZ_p$, and $\Gamma\cong\Gamma^+\times\Gamma^-\cong\ZZ_p^2$.  Note that for each Galois extension $L/K$, $c$ acts on $\Gal(L/K)$ by conjugation.  We have that $c$ acts (via conjugation) on $\Gamma^+$ as the identity and on $\Gamma^-$ by inverting each element.
Fix a topological generator $\gamma_+$ for $\Gamma^+$ and a topological generator $\gamma_-$ for $\Gamma^-$.  

For any modulus $\mathfrak{n}$, let $K(\mathfrak{n})$ be the maximal abelian extension of $K$ with conductor dividing $\mathfrak{n}$, and let $K(\mathfrak{n}^\infty\mathfrak{n}'):=\cup_mK(\mathfrak{n}^m\mathfrak{n}')$ for any coprime pair of moduli $\mathfrak{n}$ and $\mathfrak{n}'$.  Fix a modulus $\fixedconductor$ prime to $p$.  So $K_\infty\subseteq K(p^\infty)\subseteq K(p^\infty\fixedconductor)$.  Let $G = \Gal(K(p^\infty)/K) = \varprojlim_n\Gal(K(p^n)/K)$, and let $G_{\fixedconductor} = \Gal(K(p^\infty\fixedconductor)/K) = \varprojlim_n\Gal(K(p^n\fixedconductor)/K)$.    We have $G\cong\ZZ_p^\times\times\ZZ_p^\times\cong\Delta\times\Gamma,$ with $\Delta\cong\left(\ZZ/p\ZZ\right)^\times\times\left(\ZZ/p\ZZ\right)^\times$ identified with $\Gal(K(p^\infty)/K_\infty)$.  Furthermore, we have $G_\fixedconductor\cong\Delta_\fixedconductor\times\ZZ_p^\times\times\ZZ_p^\times\cong\Delta_\fixedconductor\times\Delta\times\Gamma,$ with $\Delta_\fixedconductor$ a finite group and $\Delta_\fixedconductor\times\Delta$ identified with $\Gal(K(p^\infty\fixedconductor)/K_\infty)$.  

\subsection{Characters}
We recall some key facts about Hecke characters.  Given the embedding $\sigma_\fp$, we have that for any Hecke character $\chi$ on $K$ of type $A_0$ (as in \cite{weil}), there is an associated continuous $\IC_p$-valued character (its $p$-adic avatar, obtained by shifting the component at $\infty$ to $p$), which we will also denote by $\chi$, on $G_K:=\Gal(\bar{K}/K)$.  For any Hecke character on $K$, we define $\chi^*=\chi\circ c$.  A type $A_0$ Hecke character on the quadratic imaginary field $K$ can be written in the form $\chi((a_v)_v)=\chi_f((a_v)_v)\cdot(a_{\sigma_\infty})^a(\bar{a}_{\sigma_{c\circ\sigma_\infty}})^b$ on id\`eles $(a_v)_v\in\adeles^\times_K$, with $a$ and $b$ integers and $\chi_f$ a finite order character.  The pair $(a, b)$ is the {\it infinity type} of $\chi$.  We denote by $\normcharacter$ the norm Hecke character.  The infinity type of $\normcharacter$ is $(1, 1)$.  Viewed as $\IC_p$-valued character on the Galois group, $\mathcal{N}$ factors through $\Gal((K(\mu_{p^\infty})K(\fixedconductor))/K)$, is trivial on the Galois group of the anti-cyclotomic extension, and gives the isomorphism $\Gal(K(\mu_{p^\infty})/K)\rightarrow \ZZ_p^\times$ describing the Galois-action of $\Gal(K(\mu_{p^\infty})/K)$ on $\mu_{p^\infty}$, mapping $\gamma_+$ onto a nontrivial principal unit $u\in 1+\ZZ_p\subseteq\ZZ_p$.  Note that $\Delta$ decomposes as $\Delta_+\times\Delta_-$, with $\mathcal{N}$ giving an isomorphism of $\Delta_+\subseteq\Gal(K(\mu_{p^\infty})/K)$ onto $(\ZZ/p\ZZ)^\times$ and $\mathcal{N}$ trivial on $\Delta_-$.

A continuous character $\chi$ on $G_K$ is {\it anti-cyclotomic} if $\chi(c\circ g\circ c) = \chi(g)^{-1}$ for all $g\in G_K$.  So anti-cyclotomic characters on $G_K$ are the ones that factor through anti-cyclotomic extensions.  Equivalently, a Hecke character $\chi$ is anti-cyclotomic if $\chi^* = \chi^{-1}$, so its infinity type is $(a, -a)$ for some integer $a$.  Note that for any Hecke character $\chi$, $\chi/\chi^*$ is an anti-cyclotomic character.  Also, any anti-cyclotomic character of conductor dividing $p^\infty\fixedconductor$ is trivial on $\Gal(K(\mu_{p^\infty})/K)$ and factors through the Galois group of the maximal anti-cyclotomic extension inside $\Gal(K(p^\infty\fixedconductor)/K)$.  If $\chi$ is anti-cyclotomic with infinity type not $(0, 0)$, then $\chi(\gamma_-)$ is a nontrivial element of $1+p\ZZ_p$. 

For any character $\chi$ on $G$, we write $\chi = \chi_\subGamma\chi_\subDelta$, and for any character $\chi$ on $G_{\fixedconductor}$, we write $\chi = \chi_\subGamma\chi_\subDelta\chi_{_{\Delta_{\fixedconductor}}}$, with $\chi_{_H}$ denoting the restriction of $\chi$ to $H$ for any subgroup $H$.  Note that $\normcharacter_{\subDelta_+}=\omega$, where $\omega$ denotes the Teichm\"uller character defined by $\omega(a)\equiv a\mod p$ for all $a$ in $(\ZZ/p\ZZ)^\times$.

\subsection{Iwasawa algebras}
For any pro-$p$ group $H = \varprojlim_n H_n$ with $H_n$ a finite $p$-group, define $\Lambda_H:=\ZZ_p[\![H]\!]:=\varprojlim_n\ZZ_p[H_n]$. Then we identify $\Lambda_\Gamma$ with the power series $\ZZ_p[\![T_+, T_-]\!]$ in two variables $T_+, T_-$ via
\begin{align}\label{ps-id}
\ZZ_p[\![T_+, T_-]\!]&\isomto\Lambda_\Gamma\\
T_\pm&\mapsto\gamma_\pm-\id_{\subGamma}\nonumber
\end{align}
We have $\Lambda_G = \Lambda_\Gamma[\Delta]$ and $\Lambda_{G_{\fixedconductor}} =\Lambda_\Gamma[\Delta\times\Delta_{\fixedconductor}]$.  We define $\Lambda:=\Lambda_\Gamma$.  As is conventional, given a $\ZZ_p$-algebra $R$ and a character $\chi$ on $H$, we extend $\chi$ $R$-linearly to a function on $\Lambda_H\otimes_{\ZZ_p} R$.

\section{Results on Nonvanishing}\label{NV-section}
By Equation \eqref{norm-shift}, it suffices to consider values of the $L$-function at $s=0$.  Correspondingly, $(\critical, \chi)$ is a zero of the $L$-function $L(\cdot, \pi, \cdot)$ if and only if $\chi\normcharacter^{-\critical}$ is a zero of the function $L(0, \pi, \cdot)$.  We will use these two equivalent perspectives interchangeably, both for the ($\IC$-valued) $L$-function and the $p$-adic $L$-functions.
In this section:
\begin{itemize}
\item{$\chi$ denotes a type $A_0$ Hecke character of the id\`eles $\adeles_K^\times$ of the CM field $K$}
\item{$\pi$ denotes a cuspidal automorphic representation (which is tempered) of a general unitary group preserving a Hermitian form on a vector space over $K$}
\item{$\fixedchar$ denotes a character on the finite group $\Delta\times\Delta_{\fixedconductor}$}
\end{itemize}
We will be interested in pairs $(\pi, \fixedchar)$ meeting Condition \ref{cond-pLfcnexists}:
\begin{cond}\label{cond-pLfcnexists}
There is a $p$-adic $L$-function interpolating values of $L(\measurecritical, \pi, \chi)$
as $\chi$ varies $p$-adically over Hecke characters $\chi$ such that $\chi_\subDelta\chi_{_{\Delta_{\fixedconductor}}}=\fixedchar$ with $\measurecritical$ critical for $(\pi, \chi)$.  More precisely, we require that there is an element $\mathscr{L}_{\pi, \fixedchar}\in\Lambda\otimes_{\ZZ_p}\cO$ such that for all type $A_0$ Hecke characters $\chi$ on $G_{\mathfrak{m}}$ with $\chi_\subDelta\chi_{_{\Delta_{\fixedconductor}}}=\fixedchar$ such that $\measurecritical$ is critical for $(\pi, \chi)$
\begin{align}\label{Lfcn-comparison}
\sigma_\fp^{-1}\left(c_p\left(\chi\right)\chi_{\subGamma}\left(\mathscr{L}_{\pi, \fixedchar}\right)\right) = \sigma_\infty^{-1}\left(c_\infty\left(\chi\right)L\left(\measurecritical, \pi, \chi\right)\right),
\end{align}
 with $c_\infty(\chi)$ (resp.\ $c_p(\chi)$) a product of an archimedean (resp.\ $p$-adic) period and modified Euler factors (as predicted in \cite{deligne, coates, CoPR} and shown in \cite{HELS} when $\pi$ is ordinary at $p$).  
  We also require that there exists $\chi$ with $\chi_{\Delta\times\Delta_\fixedconductor}=\fixedchar$ such that $0$ is critical for $(\pi, \chi)$, $L(0, \pi, \chi)\neq 0$, and $\chi$ is not a(n exceptional) zero of the $p$-adic $L$-function, and also that for each character $\psi$ such that the conductor of the $p$-adic avatar of the unitary part of $\psi$ is divisible by $p$, the $p$-adic $L$-function does not vanish at $\psi$ if $L(\measurecritical, \pi, \psi)\neq 0$ and $\measurecritical$ is critical for $(\pi, \psi)$.  
 (This is a very weak requirement, since for all $(\pi, \chi)$, with $\pi$ tempered, we have that $L(s, \pi, \chi)$ is nonzero for $s>\central+\frac{1}{2}$, with $\central$ the central point.  Also, we do not expect an exceptional zero at a critical point if $p$ divides the conductor of $\chi$.)  In particular, in this paper, we work under the following two assumptions:
\begin{itemize}
\item{$L(s, \pi, \chi)\neq 0$ for some character $\chi_{\Delta\times\Delta_\fixedconductor}=\fixedchar$ and some $s>s_0+\frac{1}{2}$.}
\item{$\pi$ is $p$-ordinary.}
\end{itemize}
 
When $(\pi, \fixedchar)$ meets Condition \ref{cond-pLfcnexists} for each of the (finitely many) characters $\fixedchar$ on $\Delta\times\Delta_{\fixedconductor}$, we say {\it $\pi$ meets Condition \ref{cond-pLfcnexists}}.  
\end{cond}

As noted in Section \ref{related-directions}, since the theorems (although not the propositions and lemmas) in this paper require nonvanishing at a critical point to the right of the central critical point, our theorems do not naturally extend to modular forms of weight $2$ (elliptic curves), but they may extend to certain $L$-functions associated to other data.  As mentioned in Section \ref{additional-approaches}, the case of $\GL_2$ is, however, handled via analytic methods by Rohrlich (including in \cite{rohrlich1}, which notes it is not clear how to extend those arguments for $\GL_2$ to higher rank groups).

Given the dictionary between $p$-adic measures on $\Gamma$ and elements of the Iwasawa algebra over $\Gamma$ (see e.g.\ \cite[\S 12.1-12.2]{washington}), Condition \ref{cond-pLfcnexists} is equivalent to the requirement that there is a $p$-adic measure $\mu_{\pi, \fixedchar}$ on $\Gamma$ whose values (after appropriate normalization) at $\chi$, with $\measurecritical$ critical for $(\pi, \chi)$, agree with the values of the (appropriately normalized) $L$-function $L(0, \pi, \chi)$, as in Equation \eqref{Lfcn-comparison}.

We say that $\pi$ meets Condition \ref{cond-ord} if:
\begin{cond}\label{cond-ord}
$\pi$ meets Condition \ref{cond-pLfcnexists} and the $p$-adic $L$-function has no exceptional zeroes at $(\critical, \chi)$ for points $\critical$ that are critical for $(\pi, \chi)$ and lie to the right of the central critical point for $(\pi, \chi)$.
\end{cond}

\begin{example}
Cuspidal automorphic representations that are ordinary at $p$ satisfy Condition \ref{cond-ord} by the main results of \cite{HELS}.  So we have an infinite set of representations that satisfies Condition \ref{cond-ord}.   
(The nonvanishing of the $\IC$-valued $L$-function at points $s>\central+\frac{1}{2}$ follows from the fact that we work with tempered representations.  The modified Euler factors away from $p$ for the $p$-adic $L$-function are constant volume factors, and at $p$, they are given explicitly in terms of Godement--Jacquet style zeta functions at $p$, again for tempered representations.  So for $s>\central+\frac{1}{2}$, the $p$-adic $L$-function also does not vanish.)  Note, though, that for the main theorem of the present paper, it suffices, in fact, to establish nonvanishing for one representative from each of the (finitely many) congruences classes $\bmod \left(\mathrm{lcm} (p, |\Delta_\mathfrak{m}|)\right)$, rather than for the whole set of critical points to the right of $\central+\frac{1}{2}$.

More generally, we expect a certain larger class of (tempered) cuspidal automorphic representations $\pi$ to satisfy Condition \ref{cond-ord}.  See, e.g., the Appendix.\end{example}

\begin{rmk}
We make a few observations about factorizations of Hecke characters.  Let $\chi$ be a Hecke character of infinity type $(a, b)$, and let $\anticycchar$ be an anticyclotomic character of infinity type $(1, -1)$ and conductor dividing $p^\infty$.  Then there exist Hecke characters $\chi_1$, $\chi_2$, $\chi_3$, and $\chi_4$ such that
$\chi = \chi_1\normcharacter^a$ with $\chi_1$ of infinity type $(0, b-a)$, $\chi=\chi_2\normcharacter^b$ with $\chi_2$ of infinity type $(a-b, 0)$, $\chi = \chi_3\anticycchar^a$ with $\chi_3$ of infinity type $(0,b+a)$, and $\chi = \chi_4\anticycchar^{-b}$ with $\chi_4$ of infinity type $(a+b,0)$.  More generally, any two Hecke characters of the same infinity type differ by a finite order Hecke character.  Also, if $\chi$ is of infinity type $(a, b)$, then $\chi\bar{\chi}=\normcharacter^{a+b}$.  Furthermore, if $\chi$ is of infinity type $(a, b)$ with $a\equiv b\mod 2$, we can write $\chi = \chi_f\anticycchar^{\frac{a-b}{2}}\normcharacter^{\frac{a+b}{2}}$, with $\chi_f$ a finite order character on $\Delta\times\Delta_\fixedconductor$.
\end{rmk}

Let $m$ be the exponent of the finite group $\Delta\times\Delta_\fixedconductor$.

\begin{prop}\label{all-or-none}
Suppose $(\pi, \fixedchar)$ meets condition \ref{cond-pLfcnexists}, and let $\chi$ be a type $A_0$ Hecke character of infinity type $(a, b)$ on $\Gal(K(p^\infty\fixedconductor)/K)$.
\begin{enumerate}
\item{
Fix $k_0\in\{0, \ldots, m-1\}$.  Suppose that $\measurecritical$ is critical for $(\pi, \chi)$, and let $\anticycchar$ be an anticyclotomic character whose infinity type is not $(0, 0)$.  Suppose $\left(\chi\anticycchar^{k_0}\right)_{_{\Delta\times\Delta_\fixedconductor}}=\fixedchar$.  Then either all the values  $L(\measurecritical, \pi, \chi\anticycchar^k)$, with $k\geq 0$ and $k\equiv k_0\mod m$, are zero, or only finitely many of them are zero.
}\label{prop3-generalization}
\item{Suppose $\chi_{_{\Delta\times\Delta_\fixedconductor}}=\fixedchar$.
\begin{enumerate}
\item{As $\rho$ varies over the $\IC$-valued characters of $\Gal(K_\infty^-/K)$ regarded as Dirichlet characters of $K$, either all the values $L(\measurecritical, \pi,\chi\rho)$ are zero or only finitely many of them are zero.
}\label{prop4-generalization}
\item{As $\rho$ varies over the $\IC$-valued characters of $\Gal(K_\infty^+/K)$ regarded as Dirichlet characters of $K$, either all the values $L(\measurecritical, \pi, \chi\rho)$ are zero or only finitely many of them are zero.
}\label{cyclo-generalization}
\end{enumerate}
}
\end{enumerate}
\end{prop}

\begin{rmk}\label{comparison-rmk}In the special case where we are on definite unitary groups of rank $1$ (so automorphic forms are Gr\"ossencharacters) and we fix $\pi := \normcharacter^{-1}$, $a := 1$, $b:=0$, $\fixedconductor := (1)$, $\anticycchar:=\frac{\chi}{\bar{\chi}}$, and choose $K$ so that $h(K)=1$ (where $h(K)$ denotes the class number of $K$) in Proposition \ref{all-or-none}, we are in the situation of \cite{gr1}, and:
\begin{itemize}
\item{Item \eqref{prop3-generalization} recovers the first two sentences of \cite[Proposition 3]{gr1}, i.e.\ it concerns the vanishing at the central critical point of $L(s, \chi^k)$, with $\chi$ the Hecke character associated to a CM elliptic curve.}
\item{Item \eqref{prop4-generalization} recovers the first two sentences of \cite[Proposition 4]{gr1}, i.e.\ it concerns the vanishing at the central critical point of finite order anti-cyclotomic twists of $L(s, \chi)$ with $\chi$ the Hecke character associated to a CM elliptic curve.}
\end{itemize}
\end{rmk}

\begin{proof}[Proof of Proposition \ref{all-or-none}]
For each $\cO^\times$-valued character $\psi$ on $\Gamma$ and element $f=f(\gamma_+, \gamma_-)\in\cO[\![\Gamma]\!]$, define $f_\psi\in\cO[\![\Gamma^-]\!]$ by
\begin{align*}
f_\psi(\gamma_-):=f(\psi(\gamma_+), \psi(\gamma_-)\gamma_-).
\end{align*}
So setting
\begin{align*}
\mathscr{G}:=\mathscr{L}_{\pi, \fixedchar}\in\cO[\![\Gamma]\!],
\end{align*}
and noting that since $\anticycchar$ is anti-cyclotomic, so $\anticycchar(\Gamma^+)=1$, we have
\begin{align*}
\left(\anticycchar^k\right)_{_{\Gamma^-}}\left(\mathscr{G}_\chi\right) = (\chi_{\subGamma}\anticycchar_{\subGamma}^k)(\mathscr{L}_{\pi, \fixedchar}).
\end{align*}
Via the identification \eqref{ps-id}, we identify $\mathscr{L}_{\pi, \fixedchar}$ with an element of $\cO[\![T_+, T_-]\!]$ and $\mathscr{G}_{\chi}$ with an element of $\cO[\![T_-]\!]$.
By the Weierstrass Preparation Theorem (\cite[Chapter VII Section 4]{bourbaki1998commutative}), either $\mathscr{G}_{\chi}\in\cO[\![T_-]\!]$ is the zero power series or
\begin{align}\label{WPT-equation}
\mathscr{G}_{\chi}(T_-)=p^\mu u(T_-)\mathscr{P}(T_-),
\end{align}
where $\mu\geq 0$, $u(T_-)$ is an invertible power series in $\cO[\![T_-]\!]$, and $\mathscr{P}(T_-)$ is a polynomial in $\cO[\![T_-]\!]$.  (In fact, $\mu$ can be chosen so that $\mathscr{P}(T_-)$ is {\it distinguished}, i.e.\ all but the leading coefficient of $\mathscr{P}$ lie in the maximal ideal of $\cO$, but we do not need that fact in this paper.)  So
\begin{align*}
(\chi_{\subGamma}\anticycchar^k_{\subGamma})(\mathscr{L}_{\pi, \fixedchar})&= \mathscr{G}_{\chi}\left(\anticycchar^k(\gamma_-)-1\right)\\
&= p^\mu u(\anticycchar^k(\gamma_-)-1)\mathscr{P}(\anticycchar^k(\gamma_-)-1).
\end{align*}
Since $\anticycchar(\gamma_-)$ is a nontrivial principal unit in $\ZZ_p$, we have that $|\anticycchar^k(\gamma_-)-1|_p<1$ if $k\neq 0$.  Consequently, for all $k$, the power series does not vanish on $\anticycchar^k(\gamma_-)-1$.  Since polynomials only have finitely many zeros, we therefore obtain Statement \eqref{prop3-generalization}.

For \eqref{prop4-generalization}, write $\chi = \chi'\anticycchar$ with $\anticycchar$ an anti-cyclotomic character of infinity type $(r, -r)$ for some integer $r\neq 0$, and note that since $\rho$ is anti-cyclotomic (so $\rho(\Gamma^+)=1$),
\begin{align*}
\left(\rho\anticycchar_{\Gamma^-}\right)(\mathscr{G}_{\chi'})=(\rho\chi_{\subGamma})(\mathscr{L}_{\pi, \fixedchar}).
\end{align*}
Now, $\anticycchar(\gamma_-)$ is a nontrivial principal unit in $\ZZ_p$, and $\rho(\gamma_-)$ is a $p$-power root of unity in $\cO$.  So $|\rho(\gamma_-)\anticycchar(\gamma_-)-1|_p<1$.  So similarly to the proof of Statement \eqref{prop3-generalization}, Statement \eqref{prop4-generalization} now follows from the application of the Weierstrass Preparation Theorem to $\mathscr{G}_{\chi'}$ evaluated at $\rho(\gamma_-)\anticycchar(\gamma_-)-1$ as $\rho$ ranges over finite order characters of $\Gamma^-$.

For \eqref{cyclo-generalization}, write $\chi = \chi'\normcharacter^r$ for some integer $r\neq 0$, and consider $\mathscr{G}(\chi'(\gamma_+)\gamma_+, \chi'(\gamma_-))\in\cO[\![\Gamma_+]\!]$.  Now, $\normcharacter(\gamma_-)=\rho(\gamma_-)=1$, $\normcharacter(\gamma_+)$ is a nontrivial principal unit of $\ZZ_p$, and $\rho(\gamma_+)$ ranges through $p$-power roots of unity in $\cO$.  So the remainder of the proof of Statement \eqref{cyclo-generalization} is similar to the proof of Statement \eqref{prop4-generalization}. 
\end{proof}

\begin{rmk}
The proof of Proposition \ref{all-or-none} shows that we can obtain similar statements for characters varying along any copy of $\ZZ_p$ (not just $\Gamma^+$ and $\Gamma^-$), so long as the image of a topological generator lands on a nontrivial principal unit in $\ZZ_p$.
\end{rmk}

\begin{rmk}
In the exponent of $p$ in Equation \eqref{WPT-equation}, we use $\mu$ rather than $e$ (the letter often used in the statement of the Weierstrass Preparation Theorem) to remind the reader of the connection with $\mu$-invariants.
\end{rmk}

\begin{cor}
Fix $k_0\in\{0, \ldots, m-1\}$, a positive integer $c$, and an integer $d\neq 0$ such that $\measurecritical$ is critical for $(\pi, \chi^d\normcharacter^{-c})$.  Suppose the infinity type of $\chi$ is not $(0,0)$, $\left(\chi^{d+2k_0}\normcharacter^{-k_0(a+b)-c}\right)_{_{\Delta\times\Delta_\fixedconductor}}=\fixedchar$, and $(\pi, \fixedchar)$ meets Condition \ref{cond-pLfcnexists}. Then either all the values $L(\measurecritical, \pi, \chi^{d+2k}\normcharacter^{-k(a+b)-c})$, with $k\geq 0$ and $k\equiv k_0\mod m$, are zero, or only finitely many of them are zero.
\end{cor}     
\begin{proof}
Note that $\chi^{d+2k}\normcharacter^{-k(a+b)-c}=\left(\chi^d\normcharacter^{-c}\right)\left(\frac{\chi}{\bar{\chi}}\right)^k$.  Taking $\anticycchar = \frac{\chi}{\bar{\chi}}$ in Statement \eqref{prop3-generalization} of Proposition \ref{all-or-none}, the statement follows immediately.
\end{proof}

To prove the theorems, it now suffices to prove that at least one of the central twisted values is nonzero.  In fact, Lemmas \ref{existence-lemma2} and \ref{existence-lemma} provide a stronger statement, namely that infinitely many are nonzero.

\begin{lem}\label{existence-lemma2}
Let $\chi$ and $\chi'$ be type $A_0$ Hecke characters of $\Gal(K(p^\infty\fixedconductor)/K)$ such that $\measurecritical$ is critical for $(\pi, \chi\chi')$ and $(\pi, \chi)$ and such that $\chi'$ is cyclotomic (resp.\ anti-cyclotomic) with infinity type not $(0,0)$, and suppose $L(\measurecritical, \pi, \chi\chi')\neq0$ and $(\measurecritical, \chi\chi')$ is not an exceptional zero for the $p$-adic $L$-function.  
Suppose $(\pi, \left(\chi\chi'\right)_{\Delta\times\Delta_\fixedconductor})$ meets Condition \ref{cond-pLfcnexists}.  
Then 
\begin{align*}
L(0, \pi, \chi\psi)\neq 0
\end{align*}
for infinitely many finite order cyclotomic (resp.\ anti-cyclotomic) Hecke characters $\psi$ with $\psi_{\Delta\times\Delta_{\mathfrak{m}}}=\chi'_{\Delta_\fixedconductor}$.
\end{lem}
\begin{proof}
Let $\fixedchar = (\chi\chi')_{\Delta\times\Delta_\fixedconductor}$.
Let $\phi_1, \phi_2, \ldots$ be a sequence of linear combinations of finite order cyclotomic (resp.\ anti-cyclotomic) characters converging to $\chi'$.  So since $\mathscr{L}_{\pi, \fixedchar}$ provides a $p$-adic measure,  $\lim_i\chi\phi_i(\mathscr{L}_{\pi, \fixedchar}) = \chi\chi'(\mathscr{L}_{\pi, \fixedchar})\neq 0$.  So there are infinitely many finite order cyclotomic (resp.\ anti-cyclotomic) characters $\psi_i$ occurring in the linear combinations $\phi_i$ such that $\chi\psi_i(\mathscr{L}_{\pi, \fixedchar})\neq 0$. 
\end{proof}

Note that we could instead have proved Lemma \ref{existence-lemma2} simply by applying the Weierstrass Preparation Theorem, like in the proof of Proposition \ref{all-or-none}.  Our proof of \ref{existence-lemma2} actually provides us, though, with the following stronger statement, which is much more in the spirit of (archimedean) analytic proofs of nonvanishing results, which typically rely on showing certain sums of $L$-values are nonzero in order to show that one $L$-value is nonzero.

\begin{cor}[Corollary to the proof of Lemma \ref{existence-lemma2}]\label{lincombo-cor}
Let $\chi$ be a type $A_0$ Hecke character of $\Gal(K(p^\infty\fixedconductor)/K)$ such that $\measurecritical$ is critical for $(\pi, \chi)$.  
Then
\begin{align}\label{sumequ}
\sum_\psi a_{\psi}L(0, \pi, \chi\psi)\neq 0
\end{align}
for each $\cO$-linear combination $\sum_{\psi}a_{\psi}\psi$ of finite order Hecke characters $\psi$ sufficiently close $p$-adically to $\chi$.  That is, there exists a positive integer $N$ such that Equation \eqref{sumequ} holds whenever $\left(\sum_\psi a_\psi (\gamma)\right) -\chi(\gamma)\in p^N\cO$ for all $\gamma\in \Gamma$.
\end{cor}

\begin{lem}\label{existence-lemma}
Let $k\neq 0$ be an integer, and let $\chi$ be a type $A_0$ Hecke character of $\Gal(K(p^\infty\fixedconductor)/K)$ such that $\measurecritical$ is critical for $(\pi, \chi\normcharacter^{-d})$ for $d=0, k$, $L(k, \pi, \chi)\neq 0$, and $(k, \chi)$ is not an exceptional zero for the $p$-adic $L$-function.  
Suppose $(\pi, \left(\chi\normcharacter^{-d}\right)_{\Delta\times\Delta_\fixedconductor})$ for $d=0, k$ meets Condition \ref{cond-pLfcnexists}.  
Then 
\begin{align*}
L(0, \pi, \chi\psi)\neq 0
\end{align*}
for infinitely many finite order cyclotomic Hecke characters $\psi$ with $\psi_{\Delta\times\Delta_{\mathfrak{m}}}=\omega^{-k}_\Delta\normcharacter^{-k}_{\Delta_\fixedconductor}$
\end{lem}
\begin{proof}
This follows immediately from Lemma \ref{existence-lemma2} (or from Corollary \ref{lincombo-cor}). 
\end{proof}

\begin{cor}\label{existence-cor1}
Suppose that $\pi$ meets Condition \ref{cond-pLfcnexists}.  If $\measurecritical$ is critical for $(\pi, \chi)$ and there exists an integer $k\neq 0$ and finite order cyclotomic character $\chi'$ such that  such that  $\measurecritical$ is critical for $(\pi, \chi\normcharacter^{-k})$ and  $L(k, \pi, \chi\chi')\neq 0$.
Then
\begin{align*}
L(0, \pi, \chi\psi)\neq 0
\end{align*}
for infinitely many finite order cyclotomic Hecke characters $\psi$ with $\psi_{\Delta\times\Delta_{\mathfrak{m}}}=\omega^{\ell}_\Delta\normcharacter^{\ell}_{\Delta_\fixedconductor}$ for $\ell\mod m-1$ such that $\chi'=\omega^{k-\ell}_\Delta\normcharacter^{k-\ell}_{\Delta_\fixedconductor}$.
\end{cor}
\begin{proof}
Replacing the characters $\psi_i$ from the proof of Lemma \ref{existence-lemma2} by $\chi'\psi_i$ completes the proof.
\end{proof}

\begin{cor}
Suppose that $\pi$ meets Condition \ref{cond-pLfcnexists}.
Let $\central$ be the central critical point for $(\pi, \chi)$.  If there is a critical point $\critical\neq \central$ and a finite order cyclotomic character $\chi'$ at which $L(\critical, \pi, \chi\chi')\neq 0$ and such that $(\critical, \chi\chi')$ is not an exceptional zero for the $p$-adic $L$-function, then $L(\central, \pi, \chi\psi)\neq 0$ for infinitely many finite order cyclotomic characters $\psi$ such that $\psi_{\Delta\times\Delta_{\mathfrak{m}}}=(\chi'\mathcal{N}^{-k})_{\Delta\times\Delta_{\mathfrak{m}}}$.
\end{cor}
\begin{proof}
This follows immediately from Lemma \ref{existence-lemma2} and Corollary \ref{existence-cor1}.
\end{proof}

The only way I know how to prove the existence of at least one finite order $\chi$ so that $L(\central, \pi, \chi)\neq 0$ at the central critical point $\central$ is by showing the existence of infinitely many such $\chi$.  Interestingly, Greenberg's proof in \cite{gr1} of the existence of $k\geq 0$, $k\equiv 0, 1, \ldots, p-2\mod p-1$ so that $L(\Psi^{2k+1}, k+1)\neq 0$, where $\Psi$ is the CM Grossencharacter arising from an ordinary CM elliptic curve over $\IQ$, also relies on showing the existence of infinitely many such $k$.  Both approaches, while invoking entirely different methodology, rely on relating limits (archimedean limits in his case, $p$-adic in ours) of linear combinations of the $L$-values in question to a different $L$-value known to be nonzero.
There is no clear way to extend Greenberg's approach, an intricate real-analytic argument involving Abel means, to our situation.  Likewise, there is no clear way to extend our approach, a $p$-adic argument requiring the existence of a nonvanishing critical value to the right of the central critical point, to his situation.  

\begin{rmk}
Application P.\ Monsky's results on the structure of $p$-adic power series rings, in particular \cite[Theorem 2.6]{monsky}, in our context gives a more precise description of the set of zeroes of these $p$-adic $L$-functions.  While we shall not need a more precise description in the present paper, this might be useful in attempts to combine horizontal and vertical variation to obtain still more refined nonvanishing statements.
\end{rmk}

We are now in a position to prove the main nonvanishing results.

\begin{thm}\label{ord-Thm}
Let $\pi$ be a representation meeting Condition \ref{cond-ord} (e.g.\ one of the infinitely many representations in \cite{HELS}).  For each critical point $\critical$ to the right of the central critical point $\central$ for $(\pi, \chi)$
\begin{align}\label{extra-term}
L(\central, \pi, \chi\omega^{\central-\critical}\normcharacter_{\Delta_\fixedconductor}^{\central-\critical}\psi)\neq 0
\end{align}
for all but finitely many finite order cyclotomic Hecke characters $\psi$ of $\Gamma^+$.  Moreover, for all but finitely many of those $\psi$, 
\begin{align*}
L(\central, \pi, \chi\omega^{\central-\critical}\normcharacter_{\Delta_\fixedconductor}^{\central-\critical}\psi\psi')\neq 0
\end{align*}
for all but finitely many unitary anti-cyclotomic characters $\psi'$ of $\Gamma^-$.
\end{thm}
Note that if $\fixedconductor=(1)$ (or if $|\Delta_\fixedconductor|$ divides $\central-\critical$), then $\omega^{\central-\critical}\normcharacter_{\Delta_\fixedconductor}^{\central-\critical}=\omega^{\central-\critical}$.  So if we restrict to conductor dividing $p^\infty$ (or if $|\Delta_\fixedconductor|$ divides $\central-\critical$), we can eliminate the character $\normcharacter_{\Delta_\fixedconductor}$.
\begin{proof}
By Condition \ref{cond-ord},  the theorem follows from Lemma \ref{existence-lemma}, combined with Proposition \ref{all-or-none}.   \end{proof}

All tempered cuspidal automorphic representations ordinary at $p$ satisfy Condition \ref{cond-ord}, by the main result of \cite{HELS}.  Under the weaker Condition \ref{cond-pLfcnexists}, we obtain Theorem \ref{cond-Thm}.  Note that for tempered cuspidal automorphic representations $\pi$ on unitary groups, for each $\chi$ of type $A_0$, $L(s, \pi, \chi)\neq 0$ for $s$ to the right of the central critical point for $(\pi, \chi)$.  Therefore, with information about the exceptional zeroes (equivalently, the form of the modified Euler factors) of the associated $p$-adic $L$-function, we could obtain a stronger statement.

\begin{thm}\label{cond-Thm}
For each $(\pi, \fixedchar)$ meeting Condition \ref{cond-pLfcnexists} and type $A_0$ Hecke character $\chi$ of $\adeles_K^\times$ such that $\central$ is central for $(\pi, \chi)$ and such that $\chi\normcharacter^{-\central}_{\Delta\times\Delta_\fixedconductor}=\fixedchar$, there is a finite order Hecke character $\chi'$ of $\Gamma$ such that such that at the central critical point $\central$ for $(\pi, \chi)$, 
\begin{align}\label{cond-Thm-equ}
L(\central, \pi, \chi\chi')\neq 0
\end{align}
and
\begin{align*}
L(\central, \pi, \chi\chi'\psi)\neq 0
\end{align*}for all but finitely many unitary cyclotomic (resp.\ anti-cyclotomic) characters $\psi$ of $\Gamma^+$ (resp.\ $\Gamma^-$).  Furthermore, for all but finitely many such $\psi$, $L(\central, \pi, \chi\chi'\psi\psi')\neq 0$ for all but finitely many unitary anti-cyclotomic (resp.\ cyclotomic) characters $\psi'$ of $\Gamma^-$ (resp. $\Gamma^+$).

\end{thm}
\begin{proof}
By Condition \ref{cond-pLfcnexists}, $L(\central, \pi, \chi\chi'')\neq 0$ for some Hecke character $\chi''$ on $\Gamma$ with $\measurecritical$ critical for $\chi\chi''\normcharacter^{-\central}$ and such that $\chi\chi''\normcharacter^{-\central}$ is not a zero of the $p$-adic $L$-function.  Taking limits over linear combinations of finite order characters on $\Gamma$ approaching $\chi''$, similarly to the proof of Lemma \ref{existence-lemma2}, we find a finite order character $\chi'$ occurring in the linear combinations such that $L(\central, \pi, \chi\chi')\neq 0$.  The theorem then follows from Proposition \ref{all-or-none}. 
\end{proof}

\section{Appendix: The setting of unitary groups when $p$ is unramified}\label{app-section}
From a complex analytic perspective, when working with unitary groups, the splitting condition imposed at $p$ might appear unmotivated.  This section discusses the extent to which that splitting condition can be removed, in particular sketching how to remove it in the case of signature $(n,n)$ and explaining why that argument does not fully extend to other signatures.

\subsection{The significance of the assumption that $p$ splits completely and recent progress toward removing it}\label{recent}

We begin by highlighting the ways in which the condition that $p$ splits completely manifests itself in the $p$-adic theory employed here, including in \cite{apptoSHL, HELS}:
\begin{itemize}
\item{{\it Nonemptiness of the ordinary locus.}  Hida's theory of $p$-adic automorphic forms (\cite{Hida}) is built over the ordinary locus of the moduli space over which the automorphic forms used to construct our $L$-functions lie.  When $p$ splits completely, the ordinary locus of is nonempty.  More generally, though, the nonordinary locus is empty.  Very recently, a substitute was built over the {\it $\mu$-ordinary} locus (which is nonempty whenever $p$ is unramified and agrees with the ordinary locus when $p$ splits \cite{wedhorn}).  Extension of Hida's approach to $p$-adic automorphic forms was introduced in \cite{EiMa, DSG, DSG2}, which also extend crucial earlier work on differential operators \cite{EDiffOps, EFMV} to this setting.  Further, essential results on Hecke operators and Hida's ordinary projector were extended to this setting in \cite{brasca-rosso, hernandez1}.}

\item{{\it Choice of a Siegel section at $p$ that is amenable to computation of Fourier coefficients at $p$ as well as the local zeta integrals arising from the doubling method.}  The strategy for computing local integrals arising in the doubling method is always to reduce them to integrals computed by Godement and Jacquet, which are only worked out for $\GL_n$ in \cite{jacquet} and certain other cases (neither unitary nor symplectic groups) in \cite{godement-jacquet}.\footnote{As an aside, we note that the Godement and Jacquet's assessment of their own work in the last full paragraph on \cite[p. V]{godement-jacquet} is prescient: {\it No doubt that the results developed here will someday disappear in the general theory of Euler products associated to automorphic forms.  No doubt also that this work is, at the moment, incomplete.  But we feel that its present publication could be of some use to the mathematical community.}  Indeed, other cases were later needed and developed, using \cite{godement-jacquet} as a solid foundation and inspiration.}  When the doubling method is being applied in the case of a general linear group, the reduction to the case of Godement--Jacquet is completed in \cite[Section 5]{GPSR}.  Then \cite[Section 6]{GPSR} explains how to carry out local doubling integrals for particular choices in the case of symplectic groups, by reducing to a calculation over the Levi subgroup, which is isomorphic to a general linear group, thus allowing reduction once again to Godement--Jacquet integrals for general linear groups.   

In the case where $p$ splits completely, the unitary group is isomorphic to a general linear group.  The calculations of the local zeta integrals in \cite{HELS} (which relate the zeta integrals to Euler factors of standard Langlands $L$-functions) and the calculations of the Fourier coefficients of Eisenstein series in \cite{apptoSHL} (which enables $p$-adic interpolation) both use a change of variables that relies on an isomorphism of the unitary group with a general linear group.

On the other hand, at primes that do not split completely, the unitary group is not isomorphic to a general linear group.  Computations of local doubling integrals for choices similar to those in \cite[Section 6]{GPSR} were completed for unitary groups of signature $(n,n)$ at inert primes in \cite[Section 3]{JSLi} by reduction to the computations similar to those carried out in the symplectic case in \cite[Section 6]{GPSR} (using the close relationship between $GSp_{2n}$ and unitary groups of signature $(n,n)$).

While the choices in \cite[Section 3]{JSLi} are not well-suited to $p$-adic interpolation, appropriate choices were made for symplectic groups in \cite{BS, liuJussieu, liu-rosso}.  Similarly to the approach \cite[Section 3]{JSLi}, those sections can then be adapted to the case of unitary groups of signature $(n,n)$, and continuing to use the close relationship between $GSp_{2n}$ and unitary groups of signature $(n,n)$, those sections also enable the $p$-adic interpolation of the Fourier coefficients necessary to construct the $p$-adic $L$-functions.}
\end{itemize}

\subsection{Sketch of how to extend the $p$-adic $L$-functions to inert primes $p$ for unitary groups of signature $(n,n)$ and why this approach does not immediately extend to other signatures}

The developments from Section \ref{recent} enable construction of $p$-adic $L$-functions whenever $p$ is unramified, so long as the signature of the unitary group is $(n,n)$:  As noted above (and observed in \cite[Section 3]{JSLi}), for inert primes and a unitary group $G$ signature $(n,n)$, the unitary case is sufficiently close to the symplectic case that similar choices (of local data at $p$) to those in \cite[Section 5]{liuJussieu} and \cite[Section 2.8]{liu-rosso} can be used to compute local zeta integrals arising from the doubling method, as well as to construct an Eisenstein measure on a unitary group of signature $(2n, 2n)$.  

For unitary groups $G$, there is a subtlety that does not arise for symplectic groups: We need to interpolate the pullback of an Eisenstein series constructed over $GU(2n, 2n)$ (used to construct the $L$-functions in the doubling method) to $G\times G$ (a copy of two unitary groups) inside $GU(2n, 2n)$.  The interpolation only is carried out over the ordinary locus.  As explained above, we can extend interpolation to the $\mu$-ordinary locus  (necessarily employing the associated notion of $P$-ordinariness, in place of ordinariness, as defined in \cite{HidaAsian, brasca-rosso, liu-rosso}), but the $\mu$-ordinary locus for $G\times G$ intersects the $\mu$-ordinary locus for $GU(2n, 2n)$ if and only if $G$ is of signature $(n,n)$.  Consequently, in the inert case, without substantial additional work, the above methods only extend the construction of the $p$-adic $L$-functions to signature $(n,n)$.

Each of \cite{HELS, apptoSHL, apptoSHLvv, EFMV} relies on the assumption that each prime above $p$ in the maximal totally real field $\realfield$ of the CM field $\cmfield$ splits in $\cmfield$, and in those papers, a choice of CM type for $\cmfield$ is identified with with a set of primes (one from each complex conjugate pair) of $\cmfield$ above the primes in $\realfield$ over $p$.  To set ourselves up to handle the more general case where $p$ is merely required to be unramified, we now provide careful bookkeeping that will serve as a replacement for that identification from the split case.

Since $\cmfield$ is a CM field, each embedding $\sigma: \realfield\hookrightarrow \bar{\IQ}$ extends in exactly two ways to an embedding $\cmfield\hookrightarrow\bar{\IQ}$.  We let $\Sigma$ be a CM type.  That is, $\Sigma$ is a set consisting of embeddings $\cmfield\hookrightarrow\bar{\IQ}$ such that for each $\sigma: \realfield\hookrightarrow\bar{\IQ}$, exactly one of the two extensions of $\sigma$ to $\cmfield$ is in $\Sigma$.  So, letting $\Sigma_{\realfield}$ denote the set of embeddings $\sigma: \realfield\hookrightarrow \bar{\IQ}$ (and, more generally, letting $\Sigma_L$ denote the set of embeddings $\sigma: L\hookrightarrow \bar{\IQ}$ for any number field $L$), we have a bijection
\begin{align*}
\Sigma&\rightarrow\Sigma_{\realfield}\\
\tau&\mapsto \tau|_{\realfield}.
\end{align*}
Each embedding $\iota_p\circ\tau:\cmfield\hookrightarrow\bar{\IC}_p$ determines a valuation on $\cmfield$ corresponding to a prime of $\OK$, which we denote by $\mathfrak{p}_\tau$.  That is,
\begin{align*}
\mathfrak{p}_\tau=\left\{a\in\cmfield| |\tau(a)|_p<1\right\}.
\end{align*}

Let 
\begin{align*}
\Sigma_p:=\left\{\mathfrak{p}_\tau|\tau\in\Sigma\right\}.
\end{align*}
Note that if $c$ is the nontrivial involution of $K$ in $\Gal\left(\cmfield/\realfield\right)$, then
\begin{align}\label{Galactiononprimes}
c\left({\mathfrak{p}_\tau}\right) = \mathfrak{p}_{\tau\circ c}.
\end{align}
So for each prime ideal $\mathfrak{p}$ in $\CO_{\realfield}$ over $p$, $\Sigma_p$ contains exactly one prime in $\OK$ over $\mathfrak{p}$.  Note also that
\begin{align*}
\Sigma&\rightarrow\Sigma_p\\
\tau&\mapsto \mathfrak{p}_\tau
\end{align*}
is always surjective, but it is bijective if and only if $p$ splits completely in $\realfield$.  Define
\begin{align*}
\overline{\Sigma}:=\left\{\tau\circ c| \tau\in\Sigma_{\cmfield}\right\}\\
\overline{\Sigma}_p:=\left\{c(\mathfrak{p})|\mathfrak{p}\in \Sigma_p\right\}.
\end{align*}
So
\begin{align*}
\overline\Sigma \cap \Sigma &= \emptyset\\
\overline\Sigma \cup \Sigma & = \Sigma_\cmfield\\
\overline{\Sigma}_p\cup\Sigma_p & = \left\{\mathfrak{p}|\mathfrak{p} \mbox{ is a prime ideal over $p$ in } \OK\right\}.
\end{align*}
Observe that $\Sigma_p$ is in bijective correspondence with
\begin{align*}
\Sigma_{p, \realfield}:=\left\{\mathfrak{p}|\mathfrak{p} \mbox{ is a prime ideal over $p$ in } \CO_{\realfield}\right\}
\end{align*}
via
\begin{align*}
\Sigma_p&\rightarrow\Sigma_{p, \realfield}\\
\mathfrak{p}&\mapsto \mathfrak{p}\cap\CO_{\realfield}.
\end{align*}
Likewise, we have a bijection
\begin{align*}
\overline{\Sigma}_p&\rightarrow\Sigma_{p, \realfield}\\
\mathfrak{p}&\mapsto \mathfrak{p}\cap\CO_{\realfield}.
\end{align*}

Note that $\overline{\Sigma}_p\cap\Sigma_p=\emptyset$ if and only if each prime over $p$ in $\CO_{\realfield}$ splits in $\cmfield$.  By Equation \eqref{Galactiononprimes}, we see that $\overline{\Sigma}_p$ contains the set of primes in $\OK$ determined by the embeddings in $\overline\Sigma$.  So we also have that the set of valuations determined by $\Sigma$ and the set of valuations determined by $\overline{\Sigma}$ are disjoint if and only if every prime in $\CO_{\realfield}$ over $p$ splits in $\cmfield$.

Following the convention of \cite[Section 5.2]{HELS}, let 
\begin{align*}
X_p&:=\varprojlim_r \cmfield^\times\backslash\left(K\otimes\hat\ZZ\right)^\times/U_r\\
U_r&:=\left(\CO\otimes\hat\ZZ^{(p)}\right)^\times\times\left(1+p^r\CO\otimes\ZZ_p\right)\subset\left(\cmfield\otimes\hat\ZZ\right)^\times,
\end{align*}
and identify $X_p$ with the Galois group of the maximal abelian extension of $\cmfield$ unramified away from $p$.  

Consider a Hecke character
\begin{align*}
\chi = \prod_v\chi_v: K^\times\backslash\adeles_K^\times\rightarrow\IC^\times,
\end{align*}
of type $A_0$ with infinity type $\chi_\infty=\prod_{\tau\in\Sigma}\chi_\tau$
Since $\chi$ is of type $A_0$, there are integers $k$ and $\nu_\tau$ such that each $\chi_\tau$ is of the form
\begin{align*}
\chi_\tau: \IC^\times&\rightarrow\IC^\times\\
z&\mapsto \tau(z)^{-k-\nu_\tau}{\tau(c(z))}^{\nu_\tau}.
\end{align*}
 We may write
\begin{align*}
\chi_\infty = \prod_{\mathfrak{p}\in\Sigma_p}\left(\prod_{\left\{\tau\in\Sigma|\mathfrak{p}_\tau=\mathfrak{p}\right\}}\chi_\tau\right).
\end{align*}

The component of $\chi$ at $p$ is a character
\begin{align*}
\chi_p: \left(\cmfield\otimes_{\ZZ}\ZZ_p\right)^\times\rightarrow \IC^\times,
\end{align*}
which we can decompose as a product over primes $\mathfrak{P}$ in $\cO_{\realfield}$ that divide $p$
\begin{align*}
\chi_p=\prod_{\mathfrak{P}}\chi_{\mathfrak{P}},
\end{align*}
with characters
\begin{align*}
\chi_{\mathfrak{P}}:  \left(\OK\otimes_{\ZZ}\cO_{\realfield, \mathfrak{P}}\right)^\times\rightarrow \IC^\times,
\end{align*}
where $\cO_{\realfield, \mathfrak{P}}$ denotes the completion of $\cO_{\realfield}$ at $\mathfrak{P}$.  For each prime ideal $\mathfrak{P}$ in $\cO_{\realfield}$ that divides $p$, we define
\begin{align*}
\tilde{\chi}_{\mathfrak{P}}: \left(\OK\otimes_{\ZZ}\cO_{\realfield, \mathfrak{P}}\right)^\times\rightarrow \IC_p^\times
\end{align*}
by
\begin{align*}
\tilde{\chi}_{\mathfrak{P}}:=\chi_{\mathfrak{P}}\cdot\prod_{\left\{\tau\in\Sigma|\mathfrak{p}_\tau\divides\mathfrak{P}\right\}}\tilde{\chi}_\tau
\end{align*}
where
\begin{align*}
\tilde{\chi}_\tau:  \left(\OK\otimes_{\ZZ}\cO_{\realfield, \mathfrak{P}}\right)^\times&\rightarrow \IC_p^\times
\end{align*}
is the continuous character defined for all $z\in K^\times$ relatively prime to $p$ by
\begin{align*}
z&\mapsto \tau(z)^{-k-\nu_\tau}\tau(c(z))^{\nu_\tau}.
\end{align*}

We define the $p$-adic avatar
\begin{align*}
\tilde{\chi}:X_p\rightarrow {\IC_p}^\times
\end{align*}
of $\chi$ by 
\begin{align*}
\tilde{\chi}\left(\left(a_v\right)_v\right) &:= \prod_{v\ndivides p\infty \mbox{ primes in } \OK}\chi_v(a_v)
\prod_{\mathfrak{p}\divides p \mbox{ primes in } \cO_{\realfield}}\tilde{\chi}_\mathfrak{p}\left(\left(a_{\mathfrak{p}}\right)\right)\\
\tilde{\chi}_\mathfrak{p}\left(a_\mathfrak{p}\right) &:= \chi_\mathfrak{p}\left(a_\mathfrak{p}\right)\prod_{\left\{\tau\in\Sigma|\mathfrak{p}_\tau=\mathfrak{p}\right\}}\chi_\tau\left(a_\mathfrak{p}\right).
\end{align*}

Now, with the bookkeeping from above, the choice of data and Euler factors at primes in $\Sigma_{p, \realfield}$ that split in $K$ are treated as in \cite{HELS}, while (following the observation of \cite[Section 3]{JSLi} that at inert places, the unitary group of signature $(n,n)$ can be handled similarly to a symplectic group) the choice of data and Euler factors at inert places in $\Sigma_{p, \realfield}$ are treated similarly to those in \cite[Section 5]{liuJussieu} (and, more generally, for the $P$-ordinary case, following the generalization in \cite[Section 2]{liu-rosso}).

\bibliography{OrdinaryLfcnsbib}

\begin{thebibliography}{EFMV18}

\bibitem[Art03]{arthur}
James Arthur.
\newblock The principle of functoriality.
\newblock {\em Bull. Amer. Math. Soc. (N.S.)}, 40(1):39--53, 2003.
\newblock Mathematical challenges of the 21st century (Los Angeles, CA, 2000).

\bibitem[BFH90]{BFH}
Daniel Bump, Solomon Friedberg, and Jeffrey Hoffstein.
\newblock Eisenstein series on the metaplectic group and nonvanishing theorems
  for automorphic {$L$}-functions and their derivatives.
\newblock {\em Ann. of Math. (2)}, 131(1):53--127, 1990.

\bibitem[BK90]{BK}
Spencer Bloch and Kazuya Kato.
\newblock {$L$}-functions and {T}amagawa numbers of motives.
\newblock In {\em The {G}rothendieck {F}estschrift, {V}ol.\ {I}}, volume~86 of
  {\em Progr. Math.}, pages 333--400. Birkh\"auser Boston, Boston, MA, 1990.

\bibitem[Bou98]{bourbaki1998commutative}
Nicolas Bourbaki.
\newblock {\em Commutative algebra. {C}hapters 1--7}.
\newblock Elements of Mathematics (Berlin). Springer-Verlag, Berlin, 1998.
\newblock Translated from the French, Reprint of the 1989 English translation.

\bibitem[BR19]{brasca-rosso}
R.~{Brasca} and G.~{Rosso}.
\newblock {Hida theory over some unitary Shimura varieties without ordinary
  locus}.
\newblock {\em Amer. J. Math.}, 2019.
\newblock Accepted for publication. Preprint available at
  \url{https://drive.google.com/file/d/100MYaAxzmDATss6sxGoQkb5SBKu0huQI/view}.

\bibitem[BS00]{BS}
S.~B\"ocherer and C.-G. Schmidt.
\newblock {$p$}-adic measures attached to {S}iegel modular forms.
\newblock {\em Ann. Inst. Fourier (Grenoble)}, 50(5):1375--1443, 2000.

\bibitem[CGH17]{CGH}
Frank Calegari, David Geraghty, and Michael Harris.
\newblock Bloch-{K}ato conjectures for automorphic motives.
\newblock 2017.
\newblock Submitted. \url{http://www.math.uchicago.edu/~fcale/papers/BK.pdf}.

\bibitem[CH13]{chenevier-harris}
Ga\"etan Chenevier and Michael Harris.
\newblock Construction of automorphic {G}alois representations, {II}.
\newblock {\em Camb. J. Math.}, 1(1):53--73, 2013.

\bibitem[Che04]{ch}
Ga{\"e}tan Chenevier.
\newblock Familles {$p$}-adiques de formes automorphes pour {${\rm GL}\sb n$}.
\newblock {\em J. Reine Angew. Math.}, 570:143--217, 2004.

\bibitem[Che09]{ch2}
Ga{\"e}tan Chenevier.
\newblock Une application des vari\'et\'es de {H}ecke des groupes unitaires.
\newblock 2009.
\newblock Part of Paris Book Project.
  \url{http://gaetan.chenevier.perso.math.cnrs.fr/articles/famgal.pdf}.

\bibitem[Coa91]{coates}
John Coates.
\newblock Motivic {$p$}-adic {$L$}-functions.
\newblock In {\em {$L$}-functions and arithmetic ({D}urham, 1989)}, volume 153
  of {\em London Math. Soc. Lecture Note Ser.}, pages 141--172. Cambridge Univ.
  Press, Cambridge, 1991.

\bibitem[Cog10]{cogdell}
James Cogdell.
\newblock {$L$}-functions and functoriality.
\newblock 2010.
\newblock CIMPA lectures, Weihai.
  \url{https://people.math.osu.edu/cogdell.1/cimpa-www.pdf}.

\bibitem[CPR89]{CoPR}
John Coates and Bernadette Perrin-Riou.
\newblock On {$p$}-adic {$L$}-functions attached to motives over {${\bf Q}$}.
\newblock In {\em Algebraic number theory}, volume~17 of {\em Adv. Stud. Pure
  Math.}, pages 23--54. Academic Press, Boston, MA, 1989.

\bibitem[Del79]{deligne}
P.~Deligne.
\newblock Valeurs de fonctions {$L$}\ et p\'eriodes d'int\'egrales.
\newblock In {\em Automorphic forms, representations and {$L$}-functions
  ({P}roc. {S}ympos. {P}ure {M}ath., {O}regon {S}tate {U}niv., {C}orvallis,
  {O}re., 1977), {P}art 2}, Proc. Sympos. Pure Math., XXXIII, pages 313--346.
  Amer. Math. Soc., Providence, R.I., 1979.
\newblock With an appendix by N. Koblitz and A. Ogus. ({\it English translation
  by J.\ Bieneke and E.\ Ghate available at}
  \url{http://www.math.tifr.res.in/~eghate/Deligne.pdf}).

\bibitem[Dir99]{dirichlet}
P.~G.~L. Dirichlet.
\newblock {\em Lectures on number theory}, volume~16 of {\em History of
  Mathematics}.
\newblock American Mathematical Society, Providence, RI; London Mathematical
  Society, London, 1999.
\newblock Supplements by R. Dedekind, Translated from the 1863 German original
  and with an introduction by John Stillwell.

\bibitem[dSG16]{DSG}
Ehud de~Shalit and Eyal~Z. Goren.
\newblock A theta operator on {P}icard modular forms modulo an inert prime.
\newblock {\em Res. Math. Sci.}, 3:Paper No. 28, 65, 2016.

\bibitem[dSG19]{DSG2}
E.~de~Shalit and E.~Goren.
\newblock Theta operators on unitary {S}himura varieties.
\newblock {\em Algebra and Number Theory}, 13(8):1829--1877, 2019.

\bibitem[EFMV18]{EFMV}
Ellen Eischen, Jessica Fintzen, Elena Mantovan, and Ila Varma.
\newblock Differential operators and families of automorphic forms on unitary
  groups of arbitrary signature.
\newblock {\em Documenta Mathematica}, 23:445--495, 2018.

\bibitem[EHLS20]{HELS}
E.~{Eischen}, M.~{Harris}, J.~{Li}, and C.~{Skinner}.
\newblock {$p$-adic {$L$}-functions for unitary groups}.
\newblock {\em Forum of Mathematics, Pi}, 2020.
\newblock 152 pages. Accepted for publication.

\bibitem[Eis12]{EDiffOps}
Ellen~E. Eischen.
\newblock {$p$}-adic differential operators on automorphic forms on unitary
  groups.
\newblock {\em Ann. Inst. Fourier (Grenoble)}, 62(1):177--243, 2012.

\bibitem[Eis14]{apptoSHLvv}
Ellen Eischen.
\newblock A p-adic {E}isenstein measure for vector-weight automorphic forms.
\newblock {\em Algebra Number Theory}, 8(10):2433--2469, 2014.

\bibitem[Eis15]{apptoSHL}
Ellen~E. Eischen.
\newblock A p-adic {E}isenstein measure for unitary groups.
\newblock {\em J. Reine Angew. Math.}, 699:111--142, 2015.

\bibitem[EM19]{EiMa}
E.~{Eischen} and E.~{Mantovan}.
\newblock {$p$-adic families of automorphic forms in the $\mu$-ordinary
  setting}.
\newblock {\em Amer. J. Math.}, 2019.
\newblock Accepted for publication. Preprint available at
  \url{https://www.dropbox.com/s/gx9uikwoejap6j9/muordinaryAJMrevisedfinalAugust2019.pdf?dl=0}.

\bibitem[FH95]{FrHo}
Solomon Friedberg and Jeffrey Hoffstein.
\newblock Nonvanishing theorems for automorphic {$L$}-functions on {${\rm
  GL}(2)$}.
\newblock {\em Ann. of Math. (2)}, 142(2):385--423, 1995.

\bibitem[FPR94]{fontaine-PR}
Jean-Marc Fontaine and Bernadette Perrin-Riou.
\newblock Autour des conjectures de {B}loch et {K}ato: cohomologie galoisienne
  et valeurs de fonctions {$L$}.
\newblock In {\em Motives ({S}eattle, {WA}, 1991)}, volume~55 of {\em Proc.
  Sympos. Pure Math.}, pages 599--706. Amer. Math. Soc., Providence, RI, 1994.

\bibitem[GJ72]{godement-jacquet}
Roger Godement and Herv\'e Jacquet.
\newblock {\em Zeta functions of simple algebras}.
\newblock Lecture Notes in Mathematics, Vol. 260. Springer-Verlag, Berlin-New
  York, 1972.

\bibitem[GJR04]{GJR}
David Ginzburg, Dihua Jiang, and Stephen Rallis.
\newblock On the nonvanishing of the central value of the {R}ankin-{S}elberg
  {$L$}-functions.
\newblock {\em J. Amer. Math. Soc.}, 17(3):679--722, 2004.

\bibitem[Gre83]{gr1}
Ralph Greenberg.
\newblock On the {B}irch and {S}winnerton-{D}yer conjecture.
\newblock {\em Invent. Math.}, 72(2):241--265, 1983.

\bibitem[Gre85]{gr2}
Ralph Greenberg.
\newblock On the critical values of {H}ecke {$L$}-functions for imaginary
  quadratic fields.
\newblock {\em Invent. Math.}, 79(1):79--94, 1985.

\bibitem[Har10]{harris-takagi}
Michael Harris.
\newblock Arithmetic applications of the {L}anglands program.
\newblock {\em Jpn. J. Math.}, 5(1):1--71, 2010.

\bibitem[{Her}17]{hernandez1}
V.~{Hernandez}.
\newblock {Families of Picard modular forms and an application to the
  Bloch-Kato conjecture}.
\newblock {\em ArXiv e-prints}, November 2017, 1711.03196.

\bibitem[Hid98a]{H98}
Haruzo Hida.
\newblock Automorphic induction and {L}eopoldt type conjectures for {${\rm
  GL}(n)$}.
\newblock {\em Asian J. Math.}, 2(4):667--710, 1998.
\newblock Mikio Sato: a great Japanese mathematician of the twentieth century.

\bibitem[Hid98b]{HidaAsian}
Haruzo Hida.
\newblock Automorphic induction and {L}eopoldt type conjectures for {${\rm
  GL}(n)$}.
\newblock {\em Asian J. Math.}, 2(4):667--710, 1998.
\newblock Mikio Sato: a great Japanese mathematician of the twentieth century.

\bibitem[Hid04]{Hida}
Haruzo Hida.
\newblock {\em {$p$}-adic automorphic forms on {S}himura varieties}.
\newblock Springer Monographs in Mathematics. Springer-Verlag, New York, 2004.

\bibitem[HLS06]{HLS}
Michael Harris, Jian-Shu Li, and Christopher~M. Skinner.
\newblock {$p$}-adic {$L$}-functions for unitary {S}himura varieties. {I}.
  {C}onstruction of the {E}isenstein measure.
\newblock {\em Doc. Math.}, (Extra Vol.):393--464 (electronic), 2006.

\bibitem[Jac79]{jacquet}
Herv{\'e} Jacquet.
\newblock Principal {$L$}-functions of the linear group.
\newblock In {\em Automorphic forms, representations and {$L$}-functions
  ({P}roc. {S}ympos. {P}ure {M}ath., {O}regon {S}tate {U}niv., {C}orvallis,
  {O}re., 1977), {P}art 2}, Proc. Sympos. Pure Math., XXXIII, pages 63--86.
  Amer. Math. Soc., Providence, R.I., 1979.

\bibitem[Jan17]{janu}
Fabian Januszewski.
\newblock Non-abelian $p$-adic {R}ankin-{S}elberg ${L}$-functions and
  non-vanishing of central ${L}$-values.
\newblock 2017.
\newblock arXiv preprint: 1708.02616.

\bibitem[JZ18]{jiang-zhang}
Dihua {Jiang} and Lei {Zhang}.
\newblock {On the Non-vanishing of the Central Value of Certain $L$-functions:
  Unitary Groups}.
\newblock {\em J. of European Math. Soc.}, 2018.
\newblock Accepted for publication.

\bibitem[Kat78]{kaCM}
Nicholas~M. Katz.
\newblock {$p$}-adic {$L$}-functions for {CM} fields.
\newblock {\em Invent. Math.}, 49(3):199--297, 1978.

\bibitem[Li92]{JSLi}
Jian-Shu Li.
\newblock Nonvanishing theorems for the cohomology of certain arithmetic
  quotients.
\newblock {\em J. Reine Angew. Math.}, 428:177--217, 1992.

\bibitem[Liu19]{liuJussieu}
Zheng Liu.
\newblock $p$ -adic ${L}$ -functions for ordinary families on symplectic
  groups.
\newblock {\em Journal of the Institute of Mathematics of Jussieu}, pages
  1--61, 2019.

\bibitem[LR18]{liu-rosso}
Zheng Liu and Giovanni Rosso.
\newblock Non-cuspidal {H}ida theory for {S}iegel forms and trivial zeros of
  $p$-adic {$L$}-functions.
\newblock 2018.
\newblock Preprint. Available at
  \url{https://drive.google.com/file/d/1UI2c-oL8PO_Dv9oaGuf80JPT0fCslhtp/view}.

\bibitem[Maz84]{mazur-warsaw}
B.~Mazur.
\newblock Modular curves and arithmetic.
\newblock In {\em Proceedings of the {I}nternational {C}ongress of
  {M}athematicians, {V}ol.\ 1, 2 ({W}arsaw, 1983)}, pages 185--211. PWN,
  Warsaw, 1984.

\bibitem[Mon81]{monsky}
Paul Monsky.
\newblock On {$p$}-adic power series.
\newblock {\em Math. Ann.}, 255(2):217--227, 1981.

\bibitem[OS98]{OnSk}
Ken Ono and Christopher Skinner.
\newblock Non-vanishing of quadratic twists of modular {$L$}-functions.
\newblock {\em Invent. Math.}, 134(3):651--660, 1998.

\bibitem[PSR87]{GPSR}
I.~Piatetski-Shapiro and S.~Rallis.
\newblock {\em L-functions for the classical groups}, pages 1--52.
\newblock Springer Berlin Heidelberg, Berlin, Heidelberg, 1987.

\bibitem[Roh84a]{rohrlich3}
David~E. Rohrlich.
\newblock On {$L$}-functions of elliptic curves and anticyclotomic towers.
\newblock {\em Invent. Math.}, 75(3):383--408, 1984.

\bibitem[Roh84b]{rohrlich2}
David~E. Rohrlich.
\newblock On {$L$}-functions of elliptic curves and cyclotomic towers.
\newblock {\em Invent. Math.}, 75(3):409--423, 1984.

\bibitem[Roh88]{rohrlich4}
David~E. Rohrlich.
\newblock {$L$}-functions and division towers.
\newblock {\em Math. Ann.}, 281(4):611--632, 1988.

\bibitem[Roh89]{rohrlich1}
David~E. Rohrlich.
\newblock Nonvanishing of {$L$}-functions for {${\rm GL}(2)$}.
\newblock {\em Invent. Math.}, 97(2):381--403, 1989.

\bibitem[Sha80]{shahidi}
Freydoon Shahidi.
\newblock On nonvanishing of {$L$}-functions.
\newblock {\em Bull. Amer. Math. Soc. (N.S.)}, 2(3):462--464, 1980.

\bibitem[Shi78]{shimura-hilbert}
Goro Shimura.
\newblock The special values of the zeta functions associated with {H}ilbert
  modular forms.
\newblock {\em Duke Math. J.}, 45(3):637--679, 1978.

\bibitem[SU06]{SkiUrb}
Christopher Skinner and Eric Urban.
\newblock Vanishing of {$L$}-functions and ranks of {S}elmer groups.
\newblock In {\em International {C}ongress of {M}athematicians. {V}ol. {II}},
  pages 473--500. Eur. Math. Soc., Z\"urich, 2006.

\bibitem[Was97]{washington}
Lawrence~C. Washington.
\newblock {\em Introduction to cyclotomic fields}, volume~83 of {\em Graduate
  Texts in Mathematics}.
\newblock Springer-Verlag, New York, second edition, 1997.

\bibitem[Wed99]{wedhorn}
Torsten Wedhorn.
\newblock Ordinariness in good reductions of {S}himura varieties of {PEL}-type.
\newblock {\em Ann. Sci. \'Ecole Norm. Sup. (4)}, 32(5):575--618, 1999.

\bibitem[Wei56]{weil}
Andr{\'e} Weil.
\newblock On a certain type of characters of the id\`ele-class group of an
  algebraic number-field.
\newblock In {\em Proceedings of the international symposium on algebraic
  number theory, {T}okyo \& {N}ikko, 1955}, pages 1--7, Tokyo, 1956. Science
  Council of Japan.

\end{thebibliography}

\end{document}